\documentclass[10pt, oneside]{amsart}
\title{Polynomial Valuations}

\usepackage{amssymb, amsmath, amsthm, changepage, etoolbox, geometry, mathrsfs, MnSymbol, tikz}

\newtheorem{theorem}{Theorem}[section]
\newtheorem{lemma}[theorem]{Lemma}
\newtheorem{proposition}[theorem]{Proposition}
\newtheorem{corollary}[theorem]{Corollary}

\begin{document}
\title[Polynomial Valuations on Vector Lattices]{Polynomial Valuations on Vector Lattices}

\author{Gerard Buskes}
\address{Department of Mathematics, University of Mississippi}
\email{mmbuskes@olemiss.edu}

\author{Stephan Roberts}
\address{Department of Mathematics, University of Mississippi}
\email{scrober2@olemiss.edu}

\dedicatory{Dedicated to the memory of Professor W. A. J. Luxemburg.}

\begin{abstract}
We prove that polynomial valuations on vector lattices correspond to orthosymmetric multilinear maps. As a consequence we obtain a concise proof of the equivalence of orthosymmetry and orthogonal additivity. 
\end{abstract}
\maketitle
\section{Introduction}
We use $E$,$E_1$,...,$E_n$ and $F$ to denote Archimedean vector lattices and $V$ to denote a vector space. An s-linear map $T\colon E \times ... \times E \to V$ is called \textit{orthosymmetric} if $T(x_1,...,x_s) = 0$ for all $x_1,...,x_s \in E$ for which there exist $i \neq j$ in $\{1,...,s\}$ such that $x_i \bot x_j$ (see \cite{boulabiar_vector_2006}). An s-homogeneous polynomial $P \colon E \to F$ is called \textit{orthogonally additive} if $P(x+y) = P(x) + P(y)$ whenever $x \bot y$ (see \cite{sundaresan_geometry_1991}).

The equivalence between orthosymmetric multilinear maps on vector lattices and orthogonally additive polynomials has been well studied.
The tools used to arrive at this result up to now have always included either integral representations for orthogonally additive polynomials, as  in \cite{bu_polynomials_2012} and \cite{kusraeva_representation_2011}, or notions of partitionally orthosymmetric maps, as in \cite{ben_amor_orthogonally_2015}, \cite{linares_polynomials_2009}, and \cite{loane_polynomials_2010}. In this paper we present a better tool to connect orthogonal additivity with orthosymmetry: polynomial valuations. 

We prove a straightforward correspondence between polynomial valuations and orthosymmetric multilinear maps (Theorem \ref{big equivalence with vector space range}). For good measure, we employ Hammerstein polynomials as an intermediary. Second, we use the correspondence between polynomial valuations and orthosymmetric maps to link orthogonal additivity to orthosymmetry (Theorem \ref{equivalent properties for bounded polynomials}). 

\section{Main Results}

The approach to the first result of this paper (Theorem \ref{big equivalence with vector space range}) is inspired by \cite{ercan_towards_1998}, whose authors  considered arbitrary mappings between vector lattices that satisfy the Hammerstein property. We instead consider polynomials. Let $P \colon E \to V$ be an s-homogeneous polynomial. The unique symmetric s-linear map $T \colon E \times ... \times E \to V$ for which $P(x) = T(x,...,x)$ is denoted by $\check{P}$. We say $P$ has the \textit{Hammerstein property} if $P(x+y+z)-P(x+z) = P(y+z)-P(z)$ for all $x,y,z \in E$ for which $x \bot y$. We call $P$ a \textit{polynomial valuation} if $P(x) + P(y) = P(x \wedge y) + P(x \vee y)$ for all $x,y \in E$.

We will use the following lemma to prove a correspondence between polynomials with the Hammerstein property and orthosymmetric multilinear maps.
\begin{lemma}\label{+HS iff OS lemma}
Let $s \geq 2$. Let $P \colon E \to F$ be an s-homogeneous polynomial. If $P|_{E^+}$ has the Hammerstein property, then
\[ \sum_{\delta_i = 0,1} (-1)^{s-\sum_{i=1}^s \delta_i}  P\left( x+\sum_{i=1}^s \delta_i x_i \right) =0\] 
for all $x,x_1,...,x_s \in E^+$ for which  $x_i \bot x_j$ for some $i \neq j$ in $\{1,...,s\}$.
\end{lemma}
\begin{proof}
If $s=2$, then the conclusion follows from the Hammerstein property. Assume the lemma holds for some $s \geq 2$. We prove that the lemma holds for $s+1$. We may assume $x_1 \bot x_2$. By distinguishing between $\delta_{s+1} = 1$ and $\delta_{s+1}=0$, respectively, we obtain
\begin{align*}
&\sum_{\substack{\delta_i = 0,1}} (-1)^{s+1 -\sum_{i=1}^{s+1} \delta_i}  P\left( x+\sum_{i=1}^{s+1} \delta_i x_i \right)=   \\
= &\sum_{\substack{\delta_i = 0,1  }} (-1)^{s-\sum_{i=1}^{s} \delta_i}  P\left( x_{s+1} + x + \sum_{i=1}^{s} \delta_i x_i \right) - \sum_{\substack{\delta_i = 0,1 \\}} (-1)^{s -\sum_{i=1}^s \delta_i}  P\left( x+\sum_{i=1}^{s} \delta_i x_i \right)=0
\end{align*}
where the final equality follows from applying the induction hypothesis to the above pair of sums. \end{proof}

\noindent We will now prove the connection between polynomial valuations and orthosymmetric multilinear maps. 

\begin{theorem}\label{big equivalence with vector space range}
Let $s \geq 2$ and let $P\colon E \to V$ be an s-homogeneous polynomial. The following are equivalent.
\begin{enumerate}
\item $P$ is a valuation.
\item $P|_{E^+}$ is a valuation.
\item $P$  has the Hammerstein property.
\item $P|_{E^+}$ has the Hammerstein property.
\item $\check{P}$ is orthosymmetric.
\item $\check{P}|_{E^+ \times ... \times E^+}$ is orthosymmetric.
\end{enumerate}
\end{theorem}
\begin{proof} $ $ \newline 
\noindent$(1) \Rightarrow (2)$ This implication is trivial.

\noindent$(2) \Rightarrow (4)$  Let $x,y,z \in E^+$ with $x \bot y$. If $u=x+z$ and if $v=y+z$ then
\[P(x+y+z) - P(x+z) - P(y+z) +P(z) 
= P(u \wedge v) - P(u) - P(v)  + P(u \vee v) = 0.\]

\noindent$(4) \Rightarrow (6)$ Let $x_1,...,x_s \in E^+$ with $x_i \bot x_j$ for some $i\neq j$ in $\{1,...,s\}$. Since $\check{P}$ is symmetric, we may assume $i=1$ and $j=2$. Applying the Mazur-Orlicz polarization formula (see (22) in \cite{mazur_grundlegende_1934}) and Lemma \ref{+HS iff OS lemma}, respectively, yields
\[ \check{P}(x_1,...,x_s) = \frac{1}{s!} \sum_{\delta_i = 0,1} (-1)^{s - \sum_{i=1}^s \delta_i}P\left(\sum_{i=1}^s \delta_i x_i \right) =  0. \]
Therefore, $\check{P}|_{E^+}$ is orthosymmetric.

\noindent$(6) \Rightarrow (5)$  
Let $x_1,...,x_s \in E$ with $x_i \bot x_j$ for some $i \neq j$ in $\{1,...,s\}$.  Since $\check{P}$ is symmetric, we may assume $i=1$ and $j=2$.  The orthosymmetry of $\check{P}|_{E^+ \times ... \times E^+}$ implies
\[ \check{P}(x_1,x_2,...,x_s) = \check{P}(x_1^+,x_2^+,...,x_s) - \check{P}(x_1^+,x_2^-,...,x_s) -\check{P}(x_1^-,x_2^+,...,x_s)+\check{P}(x_1^-,x_2^-,...,x_s) =0.\]
Thus, $\check{P}$ is orthosymmetric.

\noindent$(5) \Rightarrow (3)$  Let $x,y,z \in E$ with $x \bot y$. Then the binomial theorem and orthosymmetry of $P$ imply
\begin{align*}
P(x + y + z) - P(y+z) 
=& \left( \sum_{k=0}^s {{s}\choose{k}} \check{P}(\underbrace{x,...,x}_{\text{k times}},\underbrace{y+z,...,y+z}_{\text{s-k times}}) \right) -P(y+z)\\
=& \sum_{k=1}^s {{s}\choose{k}} \check{P}(\underbrace{x,...,x}_{\text{k times}},\underbrace{y+z,...,y+z}_{\text{s-k times}})\\
= &  \left( \sum_{k=1}^s {{s}\choose{k}} \check{P}(\underbrace{x,...,x}_{\text{k times}},\underbrace{z,...,z}_{\text{s-k times}})\right) +P(z)-P(z)\\
=& \left( \sum_{k=0}^s {{s}\choose{k}} \check{P}(\underbrace{x,...,x}_{\text{k times}},\underbrace{z,...,z}_{\text{s-k times}})\right)- P(z) \\
 =&P(x+z)-P(z)
\end{align*}
\noindent$(3) \Rightarrow (1)$ Combining the comment at the bottom of page 67 of \cite{ercan_towards_1998} with the proof of Corollary 2.3 of \cite{ercan_towards_1998} proves this implication.  For convenience and specificity, we copy the authors' argument. \newpage \noindent Let $x,y \in E$. Set $u = (x-y)^+$, $v=(x-y)^-$, and $w= x \wedge y$. The Hammerstein property yields\[P(x) + P(y) = P(u + w) + P(v + w) = P(u + v + w) + P(w) = P(x \vee y) + P(x \wedge y).\] 
\end{proof}

We will prove in Proposition \ref{oa polynomials are a band} that certain s-homogeneous polynomials have a valuational component. To this end, we consider multilinear maps and s-homogeneous polynomials of order bounded variation. For $x \in E^+$, a \textit{partition} of $x$ is a finite sequence of elements of $E^+$ whose sum is equal to $x$. The set of all partitions of $x$ is denoted by $\Pi x$. We write $a$ to abbreviate the partition $(a_{1},...,a_{n})$ of $x$. An s-homogeneous polynomial $P\colon E \to F$ is \textit{of order bounded variation} if \[ \left\{ \sum_k |P(a_k)| : a \in \Pi x \right\}\text{ }(x \in E^+)\] is order bounded (see \cite{buskes_roberts_arensextensions}). An s-linear map $T\colon E_1 \times ... \times E_s \to F$ is \textit{of order bounded variation} if 
\[ \left\{ \sum_{k_1,..,k_s} | T(a_{1,k_{1}},...,a_{s,k_{s}})|   :   a_1 \in \Pi x_1,..., a_s \in \Pi x_s \right\} \text{ }  (x_i \in E_i^+, i \in \{1,...,s\})\] is order bounded (see \cite{buskes_bounded_2003}). 

We denote the set of all s-linear maps of order bounded variation $E \times ... \times E \to F$ by $\mathcal{L}^{obv}(^s E;F)$ and the set of all s-homogeneous polynomials of order bounded variation $E \to F$ by $\mathcal{P}^{obv} (^s E; F)$.  By Theorem 1.1 of \cite{buskes_bounded_2003}, we have that $\mathcal{L}^{obv}(^s E;F)$ is a Dedekind complete vector lattice whenever $F$ is Dedekind complete. In addition, the authors of this paper proved $\mathcal{P}^{obv} (^s E; F)$ is a Dedekind complete vector lattice (Theorem 2.2 of \cite{buskes_roberts_arensextensions}). 

We denote the set of all orthosymmetric s-linear maps of order bounded variation $E \times ... \times E \to F$ by ${\mathcal{L}_{os}(^s E; F)}$. Van Gaans derived an interesting formula for the orthosymmetric component of a positive bilinear map in \cite{van_gaans_riesz_2001}. The following theorem proves more.

\begin{proposition}\label{os maps are a band}
Let $s \geq 2$. If $F$ is Dedekind complete, then $\mathcal{L}_{os}(^s E;F)$ is a band of $\mathcal{L}^{obv}(^s E;F).$
\end{proposition}
\begin{proof}  We first prove $\mathcal{L}_{os}(^s E;F)$ is a vector sublattice of $\mathcal{L}^{obv}(^s E;F)$.  Let $ T \in \mathcal{L}_{os}(^s E;F)$ and let $x_1,...,x_s \in E^+$ with  $x_i \bot x_j$ for some $i\neq j$ in $\{1,...,s\}$.  It follows from Theorem 1.1 of \cite{buskes_bounded_2003} that
\begin{align*}
|T|(x_1,...,x_s) &=  \sup \left\{ \sum_{k_1,...,k_n} |T(a_{1,k_1},...,a_{n,k_n})| : a_1 \in \Pi x_1,...,a_s \in \Pi x_s \right\} = 0.
\end{align*}
from which it follows that $T$ is orthosymmetric and that $\mathcal{L}_{os}(^s E;F)$ is an ideal of $\mathcal{L}^{obv}(^s E;F)$. 

Suppose $(T_\alpha)_{\alpha \in A}$ is a net in $\mathcal{L}_{os}(^s E;F)$ such that $T_\alpha \uparrow T$ in $\mathcal{L}^{obv}(^s E;F).$ Note that \\$ 0 = T_\alpha(x_1,...,x_s) \uparrow T(x_1,...,x_s)$. Therefore, $T$ is orthosymmetric. We conclude  $\mathcal{L}_{os}(^s E;F)$ is a band of $\mathcal{L}^{obv}(^s E;F)$.
\end{proof}

\noindent Ben-Amor proves the next corollary for a weaker domain but surprisingly stronger range (Theorem 6 of [2]).

\begin{corollary} \label{os maps are sym} Let $T \colon E \times ... \times E \to F$ be an s-linear map of order bounded variation. If $T$ is orthosymmetric then $T$ is symmetric.
\end{corollary}

\begin{proof}
We may assume $F$ is Dedekind complete. By Theorem \ref{os maps are a band}, there exist $T_1,T_2 \in \mathcal{L}_{os}(^s E;F)^+$ such that $T= T_1 - T_2$.  In addition, Proposition 2.1 of \cite{boulabiar_products_2003} implies $T_1$ and $T_2$ are symmetric. We conclude $T$ is symmetric.
\end{proof}

\noindent For vector lattices $E$ and $F$, we denote the set of all s-homogeneous polynomial valuations of order bounded variation $E \to F$ by $\mathcal{P}_{val}(^s E; F)$.

\begin{proposition}\label{oa polynomials are a band}
If $F$ is Dedekind complete, then $\mathcal{P}_{val} (^s E; F)$ is a band of $\mathcal{P}^{obv} (^s E; F)$.
\end{proposition}
\begin{proof}
If $s=1$, then $\mathcal{P}_{val}  (^s E; F) = \mathcal{L}^{b}(E; F)$.  Suppose $s \geq 2$. Theorem \ref{big equivalence with vector space range} implies $\mathcal{P}_{val}  (^s E; F) $ is the image of $\mathcal{L}_{os}(^s E;F)$ under the lattice isomorphism  $\check{P} \mapsto {P}$. It follows from Theorem \ref{os maps are a band} that $\mathcal{P}_{val} (^s E; F)$ is a band of  $\mathcal{P}^{obv} (^s E; F)$.
\end{proof}

Our next goal is to add orthogonal additivity to the list of equivalent properties in the statement of Theorem \ref{big equivalence with vector space range}. As a first step, we extend an orthogonally additive polynomial of order bounded variation on $E$ to the order continuous bidual of $E$.

We denote the order dual of $E$ by $E^\sim$. The \textit{Arens adjoint} of an s-linear map $T\colon E_1 \times \cdots \times E_s \to F$ is the map ${T^*}\colon F^\sim \times E_1 \times ... \times E_{s-1} \to E_s^\sim$ that is defined by
\[ T^*(f,x_1,...,x_{s-1})(x_s) = f(T(x_1,...,x_{s-1},x_s))\]
for all  $f \in F^\sim, x_1 \in E_1, ...,x_{s} \in E_{s}$ (see \cite{arens_operations_1951}). Let $T^{[1]*} := T^{*}$ and inductively define $T^{[k]*}$ by $T^{[k]*} = (T^{[k-1]*})^*$ ($k \geq 2$). For an s-homogeneous polynomial $P \colon E \to F$, we define $\bar{P}\colon E^{\sim \sim} \to F^{\sim\sim}$ by $\bar{P}(\psi) = \check P^{[s+1]*}(\psi,...,\psi)$ for all $\psi \in E^{\sim \sim}$. 

For $x \in E$, define ${\hat{x}} \in (E^\sim)^\sim_n$ by $\hat{x}(f) = f(x)$ $(f \in E^\sim)$. Let $\hat{E}:=\{ \hat{x} : x \in E\}$. We denote the order continuous component of $E^\sim$ by $E^\sim_n$. For a subset $S$ of $(E^\sim)^\sim_n$ define \[ {\mathcal{I}S}=\{x \in (E^\sim)^\sim_n : x_\alpha \uparrow x \text{ for some net } (x_\alpha)_{\alpha \in A} \text{ in } S \}\] and 
\[ {\mathcal{D}S}=\{x \in (E^\sim)^\sim_n : x_\alpha \downarrow x \text{ for some net } (x_\alpha)_{\alpha \in A} \text{ in } S \}.\] 

To denote the order convergence of $(x_\alpha)_{\alpha \in A}$ to $x$ we write $x_\alpha \rightarrow x$. An s-homogeneous polynomial $P\colon E \to F$ is called \textit{order continuous} if $P(x_\alpha) \to P(x)$ whenever $x_\alpha \to x$. We will use the following lemma to prove an orthogonally additive polynomial on $E^+$ is a polynomial valuation. 
\begin{lemma}\label{Arens OA Extension}
If $P\colon E \to F$ is an orthogonally additive s-homogeneous polynomial of order bounded variation, then $\bar{P}|_{(E^\sim)_n^{\sim+}}$ is orthogonally additive.
\end{lemma}
\begin{proof}
For readability, we denote $\bar{P}|_{(E^\sim)_n^{\sim+}}$ by $\bar{P}$. We show that $\bar{P}$ is orthogonally additive in three steps. Let $\phi, \psi \in (E^\sim)^{\sim+}$ such that $\phi \bot \psi$. 
\newline
\textbf{Step 1.} Suppose $\phi, \psi \in (I\hat{E})^+$. By definition, there exist nets $(\hat{x}_\alpha)_{\alpha \in A}$ and $(\hat{y}_\beta)_{\beta \in B}$ in $\hat E$ such that $\hat{x}_\alpha \uparrow \phi$ and $\hat{y}_\beta \uparrow \psi$. Moreover, $\phi \bot \psi$ implies $x_\alpha \bot y_\beta$ for all $\alpha \in A$ and $\beta \in B$. Then it follows from the orthogonal additivity of $P$ that
\[ \bar{P}(\hat x_\alpha + \hat y_{\beta})\phantom{x} =  \bar{P}(\hat x_\alpha) + \bar{P}(\hat y_{\beta}) \phantom{x} (\alpha \in A, \beta \in B). \]
Since $\bar{P}$ is order continuous (Theorem 3.5 of \cite{buskes_roberts_arensextensions}) we have 
$\bar{P}(\phi + \psi) =\bar{P}(\phi) + \bar{P}(\psi). $ \\
\textbf{Step 2.} Suppose $\phi, \psi \in (DI\hat{E})^+$. Then there exist nets $(f_\delta)_{\delta \in D}$ and $(g_\gamma)_{\gamma \in G}$ in $I \hat{E}$ such that $f_\delta \downarrow \phi$ and $g_\gamma \downarrow \psi$. Fix $\delta$ and $\gamma$. Since $f_\delta, g_\gamma \in I \hat{E}$, there exist nets $(\hat{x}_\alpha)_{\alpha \in A}$ and $(\hat{y}_\beta)_{\beta \in B}$ in $\hat{E}$ such that $\hat{x}_{\alpha} \uparrow f_\delta$ and $\hat{y}_{\beta} \uparrow  g_\lambda$. From the first step,
\[ \bar{P}((\hat{x}_{\alpha} - \hat{y}_{\beta})^+) + (\hat{y}_{\beta} - \hat{x}_{\alpha})^+ ) =\bar{P}((\hat{x}_{\alpha} - \hat{y}_{\beta})^+) + \bar{P}((\hat{y}_{\beta} - \hat{x}_{\alpha})^+ ). \]
Applying the order continuity of $\bar{P}$ yields
\[ \bar{P}((f_\delta - g_\gamma)^+ + (g_\gamma - f_\delta)^+ ) = \bar{P}((f_\delta - g_\gamma)^+) + \bar{P}((g_\gamma - f_\delta)^+ ) \]
and
\[ \bar{P}((\phi-\psi)^+ + (\psi - \phi)^+) = \bar{P}((\phi-\psi)^+) + \bar{P}((\psi - \phi)^+).\]
Also, $(\phi - \psi)^+ =\phi - (\psi \wedge \phi)= \phi$ and $(\psi - \phi)^+ = \psi - (\phi \wedge \psi) = \psi$. Thus $\bar{P}(\phi + \psi) = \bar{P}(\phi) + \bar{P}(\psi)$. \newpage
\noindent \textbf{Step 3.} Repeating the arguments from each of the previous steps implies
$\bar{P}(\phi + \psi) = \bar{P}(\phi) + \bar{P}(\psi)$
for all $\phi, \psi \in (DIDI\hat{E})^+$ for which $\phi \bot \psi$. Since $DIDI\hat{E} = (E^\sim)^\sim_n$ (Theorem 13 of \cite{fremlin_abstract_1967}), we conclude $\bar{P}$ is orthogonally additive. 
\end{proof}

\noindent We are now ready to prove the correspondence between orthosymmetry and orthogonal addivity.

\begin{theorem}\label{equivalent properties for bounded polynomials}
Let $s \geq 2$ and let $P\colon E \to F$ be an s-homogeneous polynomial of order bounded variation. The following are equivalent.
\begin{enumerate}
\item $P$ is orthogonally additive
\item $P|_{E^+}$ is orthogonally additive
\item $\check{P}$ is orthosymmetric.
\end{enumerate}
\end{theorem}
\begin{proof} We only need to prove $(2) \Rightarrow (3)$ and $(3) \Rightarrow (1)$. To this end, suppose $P|_{E^+}$ is orthogonally additive. By Theorem \ref{big equivalence with vector space range}, it will suffice to prove that $P|_{E^+} $ is a valuation.  Let $x,y \in E^+$. Let $G$ be the vector lattice generated by $\{x,y\}$ in $E$. Note that $(G^\sim)_n^\sim$ is nontrivial and has a unit $e := \hat{x}+\hat{y}$.  We will denote $\bar{P}|_{(G^\sim)_n^{\sim+}}$ by $\bar{P}$. 

 Let $u,v \in (G^\sim)_n^{\sim+}$ be {e-step functions}. Note that there exist disjoint components $c_1,...,c_n$ of $e$ in $G$ and $\alpha_1,...,\alpha_n,\beta_1,...,\beta_n \in \mathbb{R}^+$ such that $u = \sum_{i=1}^n \alpha_i c_i$ and $v = \sum_{i=1}^n \beta_i c_i$.

\noindent Since $\bar{P}$  is orthogonally additive (Lemma \ref{Arens OA Extension}), we have
\begin{align*}
\bar{P}(u) + \bar{P}(v) = \bar{P}\left(\sum_{i=1}^n \alpha_i c_i\right) + \bar{P}\left(\sum_{i=1}^n \beta_i c_i\right)
= \sum_{i=1}^n  (\alpha_i^s  +    \beta_i^s) \bar{P}\left(  c_i\right). 
\end{align*}
If $\alpha, \beta \in \mathbb{R}$, then $\alpha^s + \beta^s = (\alpha \wedge \beta)^s + (\alpha \vee \beta)^s$. Thus,
\begin{align*}
\sum_{i=1}^n  (\alpha_i^s  +    \beta_i^s) \bar{P}\left(  c_i \right) =&\sum_{i=1}^n  ((\alpha_i \wedge \beta_i)^s  +    (\beta_i \vee \alpha_i)^s) \bar{P}\left(  c_i\right) 
= \bar{P}(u \wedge v) + \bar{P}(u \vee v).
\end{align*}
Hence, $\bar{P}$  restricted to the $e$-step functions of $(G^\sim)_n^{\sim}$ is a valuation.

By Freudenthal's spectral theorem (Theorem 40.2 of \cite{luxemburg_riesz_1971}), there exist  sequences $(u_n)$ and $(v_m)$ of positive \textbf{e}-step functions in $(G^\sim)_n^{\sim+}$ such that $u_n \uparrow \hat{x}$ and $v_m \uparrow \hat{y}$. In addition,
\begin{align*}
\sup_n  \left( \bar{P}({u_n}) +  \bar{P}({v_n}) \right) =   \sup_n \left(\bar{P}( {u_n} \vee {v_n}) + \bar{P}({u_n} \wedge {v_n}) \right).
\end{align*} 
Since $\bar{P}$ is order continuous (Theorem 3.5 of \cite{buskes_roberts_arensextensions}), the above identity implies
\[\bar{P}(\hat{x}) +  \bar{P}(\hat{y}) = \bar{P}( \hat{x} \vee \hat{y}) + \bar{P}(\hat{x} \wedge \hat{y}).\]
That is,
\[P(x) + P(y)= P(x \vee y) + P(x \wedge y). \]
We conclude $P$ is a valuation.

Suppose $\check{P}$ is orthosymmetric. Let $x,y \in E$ with $x \bot y$. Then the binomial theorem and orthosymmetry of $\check{P}$ imply
\[P(x+y) = \sum_{k=0}^s {{s}\choose{k}} \check{P}(\underbrace{x,...,x}_{\text{k times}},\underbrace{y,...,y}_{\text{s-k times}}) = {P}(x) + {P}(y). \]
Therefore $P$ is orthogonally additive.
\end{proof}

The following consequence of Proposition \ref{oa polynomials are a band} and Theorem \ref{equivalent properties for bounded polynomials} generalizes Theorem 6.2 of \cite{bu_polynomials_2012}.

\begin{corollary}
If $F$ is Dedekind complete, then the set of all orthogonally additive s-homogeneous polynomials of order bounded variation $E \to F$ is a Dedekind complete vector lattice.
\end{corollary}

\noindent We conclude with an application of Theorem \ref{equivalent properties for bounded polynomials} that implies Theorem 2.2 of \cite{bernau_order_1995}.

\begin{corollary} \label{orthsymnormalpart}
Let $s \geq 2$ and let  $T\colon E \times ... \times E \to F$ be an s-linear map of order bounded variation. If $T$ is orthosymmetric, then the restriction of $T^{[s+1]*}$ to $(E^\sim)^\sim_n \times ... \times (E^\sim)^\sim_n $ is orthosymmetric. 
\end{corollary}
\begin{proof} 
Let $P_T$ be the s-homogeneous polynomial that is generated by $T$. That is, $P_T(x) = T(x,...,x)$ for all $x \in E$. For convenience, we also use $\overline{P_T}$ to denote the restriction of $\overline{P_T}$ to $(E^\sim)_n^\sim$.  Since $P_T$ is an orthogonally additive polynomial of order bounded variation, Lemma \ref{Arens OA Extension} implies $\overline{P_T}$ is orthogonally additive. Then Theorem \ref{equivalent properties for bounded polynomials} implies $(\overline{P_T})\check{\vphantom{P}}$ , the unique symmetric s-linear map that generates  $\overline{P_T}$, is orthosymmetric. In addition, Corollary \ref{os maps are sym} implies $T$ is symmetric. Hence, $T^{[s+1]*}$ restricted to $(E^\sim)^\sim_n \times ... \times (E^\sim)^\sim_n $ is symmetric (Theorem 3.4 of \cite{buskes_roberts_arensextensions}). It follows from the uniqueness of $(\overline{P_T})\check{\vphantom{P}}$  that  $T^{[s+1]*} = (\overline{P_T})\check{\vphantom{P}}$ on $(E^\sim)^\sim_n \times ... \times (E^\sim)^\sim_n $. We conclude  $T^{[s+1]*}$ restricted to $(E^\sim)^\sim_n \times ... \times (E^\sim)^\sim_n $ is orthosymmetric.
\end{proof}

\newpage
\bibliographystyle{amsplain}
\bibliography{myreferences}
\end{document}